\documentclass[11pt]{amsart}

\usepackage[utf8]{inputenc}
\usepackage[T1]{fontenc}
\usepackage{amsmath,amssymb,amsthm,mathtools}
\usepackage{enumitem}
\usepackage[margin=1.15in]{geometry}
\usepackage[dvipsnames]{xcolor}
\usepackage[colorlinks=true,linkcolor=black,citecolor=black,urlcolor=black]{hyperref}

\newtheorem{theorem}{Theorem}[section]
\newtheorem{lemma}[theorem]{Lemma}
\newtheorem{proposition}[theorem]{Proposition}
\newtheorem{corollary}[theorem]{Corollary}
\newtheorem{introthm}{Theorem}

\theoremstyle{definition}
\newtheorem{definition}[theorem]{Definition}
\newtheorem{example}[theorem]{Example}

\theoremstyle{remark}
\newtheorem{remark}[theorem]{Remark}

\numberwithin{equation}{section}
\newcommand{\C}{\mathbb{C}}
\newcommand{\Z}{\mathbb{Z}}
\newcommand{\PP}{\mathbb{P}}
\newcommand{\OO}{\mathcal{O}}
\newcommand{\I}{\mathcal{I}}
\newcommand{\Pcal}{\mathcal{P}}
\newcommand{\Ah}{\widehat{A}}
\newcommand{\m}{\mathfrak{m}}
\newcommand{\NS}{\operatorname{NS}}
\newcommand{\Pic}{\operatorname{Pic}}
\newcommand{\Ext}{\operatorname{Ext}}
\newcommand{\Hom}{\operatorname{Hom}}
\newcommand{\End}{\operatorname{End}}
\newcommand{\Tor}{\operatorname{Tor}}
\newcommand{\rk}{\operatorname{rk}}
\newcommand{\ch}{\operatorname{ch}}
\newcommand{\Sym}{\operatorname{Sym}}
\newcommand{\Spec}{\operatorname{Spec}}

\newcommand{\length}{\operatorname{length}}
\newcommand{\codim}{\operatorname{codim}}

\newcommand{\pd}{\operatorname{pd}}
\newcommand{\depth}{\operatorname{depth}}
\newcommand{\id}{\mathrm{id}}
\newcommand{\Db}{D^{b}}
\newcommand{\R}{\mathbf{R}}

\begin{document}

\title[Line-filtered ACM bundles and Ulrich defects]{Line-filtered ACM bundles and Ulrich defects\\ on abelian varieties of Picard rank one}

\author{Soham Mondal}
\address{Department of Mathematics, Indian Institute of Science Education and Research Berhampur, India}
\email{getsoham1@gmail.com}


\date{\today}

\begin{abstract}
Let $A$ be a complex abelian variety with $\NS(A)=\Z[L]$, $L$ ample. We show that a vector bundle filtered by line bundles splits uniquely into unipotent blocks, determine when it is ACM with respect to $L$, and show that it is never Ulrich. For a rank-two Ulrich bundle we bound from below the degree of the defect of every saturated line subbundle, compute the Hilbert polynomial of the defect, and show that the bound is attained exactly when an explicit cohomological support locus is nonempty. On abelian surfaces this locus is an ample curve in $\Pic^0(A)$; as a consequence, every rank-two Ulrich bundle is a Serre extension of a fixed numerical type, which we characterize and construct for every ample $L$. For $g\ge3$ we reduce the equality case to the existence of subcanonical codimension-two subschemes with prescribed cohomology.
\end{abstract}

\maketitle

\section{Introduction}

Let $X$ be a smooth projective variety of dimension $g\ge 2$ and $L$ an ample line bundle on $X$. A vector bundle $E$ on $X$ is \emph{arithmetically Cohen--Macaulay} (ACM) with respect to $L$ if
\[
H^i(X,E\otimes L^t)=0\qquad\text{for all } t\in\Z \text{ and } 1\le i\le g-1,
\]
and it is \emph{Ulrich} if $H^\bullet(X,E\otimes L^{-p})=0$ for $1\le p\le g$. When $L$ is very ample these are the familiar conditions on the section module of $E$ \cite{ES03}; Ulrich bundles are exactly the ACM bundles whose module has the maximal number of generators. Beauville \cite{Bea16} showed that every abelian surface carries a rank-two Ulrich bundle, and Coskun, Nuer and Yoshioka \cite{CNY25} determined the Chern characters of Ulrich bundles on abelian surfaces. In dimension at least three very little is known about low-rank Ulrich bundles on abelian varieties.

Abelian varieties are a natural place to test these notions because the structure sheaf is not ACM when $g\ge2$: $h^1(A,\OO_A)=g$. On the other hand every nontrivial line bundle in $\Pic^0(A)$ is acyclic, and so ACM for every polarization \cite{MB26}. This paper studies two questions on an abelian variety $A$ of dimension $g\ge2$ with
\[
\NS(A)=\Z[L],\qquad L \text{ ample}.
\]
\begin{enumerate}[label=(\arabic*)]
\item Which bundles admitting a filtration by line bundles are ACM, and can any of them be Ulrich?
\item A rank-two bundle $E$ with a saturated line subbundle $N$ has a quotient $M\otimes\I_Z$, where $Z$ is a codimension-two scheme which we call the \emph{defect} of $N\subset E$. What does the Ulrich condition force on $N$ and on $Z$?
\end{enumerate}

Every line bundle on $A$ is uniquely of the form $L^a\otimes P$ with $a\in\Z$ and $P\in\Pic^0(A)$. A vector bundle is \emph{line-filtered} if it has a filtration by subbundles whose quotients are line bundles, and \emph{unipotent} if all these quotients are $\OO_A$.

\begin{introthm}[Theorem~\ref{thm:blocks}, Proposition~\ref{prop:HN}, Corollary~\ref{cor:no-ulrich}]\label{thmA}
Every nonzero line-filtered bundle $E$ on $A$ decomposes as
\[
E\;\cong\;\bigoplus_{(a,P)} L^a\otimes P\otimes U_{a,P},
\]
over finitely many distinct pairs $(a,P)\in\Z\times\Pic^0(A)$, with each $U_{a,P}$ a nonzero unipotent bundle. The summands $L^a\otimes P\otimes U_{a,P}$ are unique up to isomorphism, and $\rk U_{a,P}$ equals the number of graded pieces isomorphic to $L^a\otimes P$ in any line filtration of $E$. The bundle $E$ is ACM with respect to $L$ if and only if every occurring $P$ is nontrivial. It is slope-semistable if and only if only one value of $a$ occurs; it is never slope-stable when $\rk E\ge2$, and never Ulrich.
\end{introthm}

The ACM criterion rests on a nonvanishing statement: every cohomology group $H^i(A,U)$, $0\le i\le g$, of a nonzero unipotent bundle $U$ is nonzero (Lemma~\ref{lem:unipotent-nonvanishing}). In rank two, Theorem~\ref{thmA} gives a short normalized list (Theorem~\ref{thm:rank-two}); the indecomposable members are nonsplit self-extensions of a nontrivial $P\in\Pic^0(A)$, and under the Fourier--Mukai transform they correspond to nonreduced length-two subschemes of the dual abelian variety supported away from the origin (Proposition~\ref{prop:length-two}).

By Theorem~\ref{thmA}, a rank-two Ulrich bundle on $A$ is never line-filtered, so each of its saturated line subbundles has a nonempty defect. Put $H=c_1(L)$ and $d=H^g/g!=h^0(A,L)$.

\begin{introthm}[Theorem~\ref{thm:defect-bound}, Proposition~\ref{prop:hilbert}]\label{thmB}
Let $E$ be a rank-two Ulrich bundle on $(A,L)$. Then $E$ is $\mu_L$-stable and $c_1(E)\equiv(g+1)H$. For every saturated exact sequence
\[
0\to L^a\otimes P\to E\to L^b\otimes Q\otimes\I_Z\to 0,\qquad P,Q\in\Pic^0(A),
\]
one has $a+b=g+1$ and $a\le1$; the defect $Z$ is a nonempty local complete intersection of pure codimension two with $\omega_Z\cong(L^{b-a}\otimes Q\otimes P^{-1})|_Z$,
\[
[Z]\cdot H^{g-2}=\Big(\frac{(g+1)(3g+4)}{12}-ab\Big)H^g\;\ge\;\frac{3g^2-5g+4}{12}\,H^g,
\]
and
\[
\chi\big(\OO_Z\otimes L^t\big)=d\Big[t^g-2\prod_{k=1}^{g}(t+a-k)+(t+2a-g-1)^g\Big].
\]
Equality in the degree bound holds exactly when $a=1$, and then $P$ and $Q$ are nontrivial.
\end{introthm}

The degree bound uses only the Ulrich Hilbert polynomial and the vanishing $H^0(A,E\otimes L^{-1})=0$. For $g=2$ it is attained on every $(A,L)$ (Theorem~\ref{thmD}); for $g\ge3$ it is not known whether rank-two Ulrich bundles exist on $(A,L)$. The next result decides when the bound is attained. For a coherent sheaf $F$ on $A$ write
\[
V^i(F)=\{R\in\Pic^0(A)\;:\;H^i(A,F\otimes R)\neq0\}
\]
for the cohomological support loci in the sense of Green and Lazarsfeld \cite{GL87}.

\begin{introthm}[Theorem~\ref{thm:jump}]\label{thmC}
Let $E$ be a rank-two Ulrich bundle on $(A,L)$. The minimal-defect subsheaves of $E$ are exactly the line bundles $L\otimes R^{-1}$ with $R\in V^0(E\otimes L^{-1})$; for such $R$ every nonzero map $L\otimes R^{-1}\to E$ is a saturated embedding. The locus $V^0(E\otimes L^{-1})$ is a proper closed subset of $\Pic^0(A)$. If it is empty, then every saturated line subbundle has $a\le0$, every defect satisfies $[Z]\cdot H^{g-2}\ge\frac{(g+1)(3g+4)}{12}H^g$, and $V^i(E\otimes L^{-1})\neq\emptyset$ for some $1\le i\le g-1$.
\end{introthm}

When $g=2$ there is no room for intermediate cohomology except in degree one, and the Euler characteristic forces $V^0(E\otimes L^{-1})\ne\emptyset$. Combined with an Ulrich criterion for extensions (Theorem~\ref{thm:surface-criterion}), this gives a description of all rank-two Ulrich bundles on abelian surfaces of Picard rank one.

\begin{introthm}[Theorems~\ref{thm:surface-classification} and~\ref{thm:surface-existence}, Proposition~\ref{prop:divisor}]\label{thmD}
Let $A$ be an abelian surface with $\NS(A)=\Z[L]$, $L$ ample, $L^2=2d$.
\begin{enumerate}[label=\textup{(\arabic*)}]
\item A rank-two vector bundle $E$ on $A$ is Ulrich if and only if there are nontrivial $P,Q\in\Pic^0(A)$, a zero-dimensional subscheme $Z\subset A$ of length $d$, and an exact sequence
\[
0\to L\otimes P\to E\to L^2\otimes Q\otimes\I_Z\to0
\]
such that the restriction maps $H^0(A,L\otimes Q)\to H^0(\OO_Z)$ and $H^0(A,L\otimes P^{-1})\to H^0(\OO_Z)$ are isomorphisms.
\item If $E$ is Ulrich, then $V^0(E\otimes L^{-1})$ is the support of an ample effective divisor $D_E\subset\Pic^0(A)$ with $D_E^2=2d$, and the subsheaves $L\otimes P$ occurring in \textup{(1)} are exactly the $L\otimes R^{-1}$ with $R\in D_E$.
\item For nontrivial $P,Q\in\Pic^0(A)$, a curve $C\in|L\otimes Q\otimes P^{-1}|$ and general points $x_1,\dots,x_d\in C$, a locally free extension as in \textup{(1)} with $Z=\{x_1,\dots,x_d\}$ exists, and every such extension is a $\mu_L$-stable Ulrich bundle.
\end{enumerate}
\end{introthm}

Thus every rank-two Ulrich bundle on such a surface attains the bound of Theorem~\ref{thmB}, and its minimal-defect subsheaves form a one-dimensional family parametrized by a curve. Existence for very ample $L$ is due to Beauville \cite{Bea16}, and for arbitrary ample $L$ it also follows from \cite[Theorem~1.10]{CNY25} (Remark~\ref{rem:CNY}); part (3) produces the bundles explicitly, including for principal polarizations (Corollary~\ref{cor:principal}).

For $g\ge3$ the following result turns the equality case of Theorem~\ref{thmB} into a question about subvarieties of $A$.

\begin{introthm}[Theorem~\ref{thm:higher-criterion}, Corollary~\ref{cor:HS}]\label{thmE}
Let $g\ge3$. There is a rank-two Ulrich bundle $E$ on $(A,L)$ with $V^0(E\otimes L^{-1})\ne\emptyset$ if and only if there are nontrivial $P,Q\in\Pic^0(A)$ and a nonempty local complete intersection $Z\subset A$ of pure codimension two with $\omega_Z\cong(L^{g-1}\otimes Q\otimes P^{-1})|_Z$ such that, for $R\in\{Q,P^{-1}\}$,
\begin{enumerate}[label=\textup{(\alph*)}]
\item $H^j(Z,R|_Z)=0$ for $0\le j\le g-3$, and
\item for $1\le k\le g-1$, the restriction map $H^0(A,L^k\otimes R)\to H^0(Z,L^k\otimes R|_Z)$ is an isomorphism and $H^j(Z,L^k\otimes R|_Z)=0$ for $1\le j\le g-3$.
\end{enumerate}
\end{introthm}

For $g=3$, $Z$ is a curve of degree $8d$ and arithmetic genus $8d+1$, and the conditions say that $Q$ and $P^{-1}$ have no nonzero sections on $Z$ and that four restriction maps are isomorphisms (Corollary~\ref{cor:g3}).

\subsection*{Relation with earlier work}
The cohomology of line bundles on abelian varieties and the structure of homogeneous and unipotent bundles are classical \cite{Muk78,Muk81,Mum70,BL04}. In particular the decomposition of a numerically flat bundle as $\bigoplus P\otimes U_P$ is the classical description of homogeneous bundles due to Matsushima, Morimoto, Miyanishi and Mukai; see \cite[\S3]{Muk81}; Proposition~\ref{prop:numflat} is included because the ACM criterion is needed for every polarization. A recent account of semihomogeneous bundles and their moduli is \cite{GKUW26}.

Ballico \cite[Theorem~1.1]{Bal08} classified rank-two totally ACM bundles on abelian surfaces with $\operatorname{Num}\cong\Z$, a condition that tests all ample line bundles at once and assumes no filtration. Our ACM results fix the polarization, assume a line filtration, and hold in every dimension $g\ge2$; Section~\ref{sec:polarization} identifies Ballico's family inside our list. Yoshioka \cite{Yos21} studies semistable ACM sheaves on general abelian surfaces, and Coskun--Nuer--Yoshioka \cite{CNY25} classify the Chern characters of Ulrich bundles on abelian surfaces via weak Brill--Noether; when $\NS(A)=\Z[L]$ their Theorem~1.10 also gives the existence of rank-two Ulrich bundles for every ample $L$. Theorem~\ref{thmD} is of a different nature: it describes all rank-two Ulrich bundles by one type of Serre extension and constructs them explicitly.

The existence of indecomposable ACM bundles in every rank, filtered by a fixed nontrivial $P\in\Pic^0(A)$, is proved in joint work with Barik \cite{MB26}; Proposition~\ref{prop:curvilinear} recalls these bundles and computes their self-extensions.

\subsection*{Organization}
Section~\ref{sec:prelim} collects facts on line bundles. Section~\ref{sec:blocks} proves the block decomposition and the stability statements. Section~\ref{sec:rank-two} treats rank-two blocks and their Fourier--Mukai parameters. Section~\ref{sec:defects} studies saturated line subbundles and their defects. Section~\ref{sec:ulrich} contains the results on Ulrich bundles, and Section~\ref{sec:polarization} discusses nonprimitive polarizations and the totally ACM condition.

\subsection*{Conventions}
We work over $\C$. A vector bundle is identified with its locally free sheaf. Unless stated otherwise, $A$ is an abelian variety of dimension $g\ge2$ with $\NS(A)=\Z[L]$, $L$ ample; in particular $L$ is primitive. We write $H=c_1(L)$, $d=H^g/g!$, and $E(t)=E\otimes L^t$. A sheaf is \emph{acyclic} if all of its cohomology groups vanish. Identities between Chern classes are taken modulo numerical equivalence, so degree-zero line bundles contribute nothing to them. For a pure codimension-two scheme $Z$, $[Z]$ denotes its fundamental cycle, and $[\emptyset]=0$. The slope of a torsion-free sheaf $F$ of positive rank is $\mu_L(F)=c_1(F)\cdot H^{g-1}/\rk F$.

\section{Preliminaries on line bundles}\label{sec:prelim}

The exact sequence $0\to\Pic^0(A)\to\Pic(A)\to\NS(A)\to0$ and $\NS(A)=\Z[L]$ show that every line bundle is isomorphic to $L^a\otimes P$ for a unique $a\in\Z$ and a unique $P\in\Pic^0(A)$.

\begin{lemma}\label{lem:line-cohomology}
Let $P\in\Pic^0(A)$ and $s\in\Z$.
\begin{enumerate}[label=\textup{(\roman*)}]
\item $H^i(A,\OO_A)\cong\bigwedge^iH^1(A,\OO_A)$ has dimension $\binom gi$ for $0\le i\le g$.
\item If $P\not\cong\OO_A$, then $P$ is acyclic.
\item If $s>0$, then $H^i(A,L^s\otimes P)=0$ for $i>0$ and $h^0(A,L^s\otimes P)=s^gd$.
\item If $s<0$, then $H^i(A,L^s\otimes P)=0$ for $i<g$ and $h^g(A,L^s\otimes P)=|s|^gd$.
\end{enumerate}
In all cases $\chi(A,L^s\otimes P)=s^gd$.
\end{lemma}

\begin{proof}
Parts (i) and (ii) are standard \cite[\S8, \S13]{Mum70}, \cite[Ch.~1--3]{BL04}. For $s>0$ the line bundle $L^s\otimes P$ is ample and $\omega_A\cong\OO_A$, so Kodaira vanishing gives (iii) once we know $\chi$. Since the tangent bundle of $A$ is trivial, $\operatorname{td}(A)=1$ and Riemann--Roch gives $\chi(L^s\otimes P)=(sH)^g/g!=s^gd$. Part (iv) follows from (iii) by Serre duality.
\end{proof}

More generally, $\chi(A,F\otimes R)=\chi(A,F)$ for every coherent sheaf $F$ and every $R\in\Pic^0(A)$, since $c_1(R)$ is numerically trivial.

\begin{corollary}\label{cor:ACM-lines}
The line bundle $L^a\otimes P$ is ACM with respect to $L$ if and only if $P\not\cong\OO_A$.
\end{corollary}

\begin{proof}
By Lemma~\ref{lem:line-cohomology}, the twists $L^{a+t}\otimes P$ with $a+t\ne0$ have no intermediate cohomology. The twist with $t=-a$ is $P$, which is acyclic if $P\ne\OO_A$ and has $h^1=g>0$ if $P=\OO_A$. Since $g\ge2$, degree one is intermediate.
\end{proof}

\begin{lemma}\label{lem:ext-ACM}
An extension of two ACM bundles is ACM.
\end{lemma}

\begin{proof}
Twist the extension by $L^t$ and use the long exact cohomology sequence in each degree $1\le i\le g-1$.
\end{proof}

\section{Unipotent blocks and stability}\label{sec:blocks}

\subsection{The Fourier--Mukai transform}
Let $\Ah=\Pic^0(A)$ be the dual abelian variety and $\Pcal$ the normalized Poincar\'e bundle on $A\times\Ah$, so that $\Pcal|_{\{0\}\times\Ah}\cong\OO_{\Ah}$. For $\alpha\in\Ah$ put $P_\alpha=\Pcal|_{A\times\{\alpha\}}$. The map $\alpha\mapsto P_\alpha$ is a group isomorphism $\Ah\cong\Pic^0(A)$. Let $t_\alpha\colon\Ah\to\Ah$ be the translation $y\mapsto y+\alpha$, and let $p_A$, $p_{\Ah}$ be the projections. The Fourier--Mukai functor
\[
\Phi=\R p_{\Ah*}\big(p_A^*(-)\otimes\Pcal\big)\colon \Db(A)\to\Db(\Ah)
\]
is an equivalence of triangulated categories, and $\Phi(\OO_A)\cong k(0)[-g]$ \cite[Theorem~2.2]{Muk81}, \cite[\S13]{Mum70}.

\begin{lemma}\label{lem:FM-twist}
For every $\alpha\in\Ah$ and $F\in\Db(A)$ we have $\Phi(P_\alpha\otimes F)\cong t_\alpha^*\Phi(F)$. In particular $\Phi(P_\alpha)\cong k(-\alpha)[-g]$, and the point $-\alpha$ parametrizes $P_\alpha^{-1}$.
\end{lemma}

\begin{proof}
We first show $p_A^*P_\alpha\otimes\Pcal\cong(1_A\times t_\alpha)^*\Pcal$. Both sides restrict to $P_{\alpha+\beta}$ on $A\times\{\beta\}$, and both are trivial on $\{0\}\times\Ah$ because $\Pcal$ is normalized; the seesaw principle gives the isomorphism. Put $G=p_A^*F\otimes\Pcal$. Since $p_A\circ(1_A\times t_\alpha)=p_A$, we get $p_A^*(P_\alpha\otimes F)\otimes\Pcal\cong(1_A\times t_\alpha)^*G\cong(1_A\times t_{-\alpha})_*G$. As $p_{\Ah}\circ(1_A\times t_{-\alpha})=t_{-\alpha}\circ p_{\Ah}$,
\[
\Phi(P_\alpha\otimes F)\cong\R p_{\Ah*}(1_A\times t_{-\alpha})_*G\cong t_{-\alpha*}\Phi(F)=t_\alpha^*\Phi(F).
\]
For $F=\OO_A$ this gives $\Phi(P_\alpha)\cong t_\alpha^*k(0)[-g]=k(-\alpha)[-g]$.
\end{proof}

By flat base change, for every $\beta\in\Ah$ and $F\in\Db(A)$, the derived fibre of $\Phi(F)$ at $\beta$ is $\R\Gamma(A,F\otimes P_\beta)$. We shall use this in the following form.

\begin{lemma}\label{lem:fibres}
If $F\in\Db(A)$ satisfies $\R\Gamma(A,F\otimes R)=0$ for every $R\in\Pic^0(A)$, then $F=0$.
\end{lemma}

\begin{proof}
Put $K=\Phi(F)$, a bounded complex with coherent cohomology. If $K\ne0$, let $m$ be the largest integer with $\mathcal H^m(K)\neq0$ and choose $\beta$ in its support. By right exactness of the derived fibre in the top degree, $H^m(K\otimes^{\mathbf L}k(\beta))\cong\mathcal H^m(K)\otimes k(\beta)\neq0$ by Nakayama's lemma. This contradicts the vanishing of the derived fibre $\R\Gamma(A,F\otimes P_\beta)$. Hence $\Phi(F)=0$, and $F=0$ because $\Phi$ is an equivalence.
\end{proof}

\subsection{Nonvanishing for unipotent bundles}

\begin{lemma}\label{lem:unipotent-nonvanishing}
If $U$ is a nonzero unipotent bundle on $A$, then $H^i(A,U)\neq0$ for every $0\le i\le g$.
\end{lemma}

\begin{proof}
Applying $\Phi$ to a unipotent filtration of $U$ and using $\Phi(\OO_A)=k(0)[-g]$, the long exact sequence of cohomology sheaves shows inductively that $\Phi(U)\cong T[-g]$ (see also \cite[Example~2.9]{Muk81}), where $T$ is a sheaf of finite length $\rk U$ supported at $0$. Let $R=\OO_{\Ah,0}$, a regular local ring of dimension $g$ with residue field $k$. Since $\Phi$ is fully faithful and $T$ is supported at a point,
\[
H^i(A,U)=\Ext^i_A(\OO_A,U)\cong\Hom_{\Db(\Ah)}\big(k(0)[-g],T[-g+i]\big)=\Ext^i_{\Ah}(k(0),T)=\Ext^i_R(k,T).
\]
The Koszul complex $K_\bullet$ on a regular system of parameters resolves $k$ and satisfies $\Hom_R(K_\bullet,R)\cong K_{g-\bullet}$, so $\Ext^i_R(k,T)\cong\Tor^R_{g-i}(k,T)$. Let $F_\bullet$ be a minimal free resolution of $T$. As $T\ne0$ has finite length, $\depth T=0$ and $\pd_RT=g$ by the Auslander--Buchsbaum formula \cite[Theorem~19.9]{Eis95}.\linebreak[2] No $F_j$ with $0\le j\le g$ vanishes: if $F_j=0$ with $j<g$, exactness at $F_{j+1}$ makes $F_{j+1}\to F_j$ zero, so $F_{j+2}\to F_{j+1}$ is surjective, while minimality puts its image in $\m F_{j+1}$; Nakayama gives $F_{j+1}=0$, and inductively $F_g=0$, a contradiction. Since the differentials of $F_\bullet\otimes k$ vanish, $\Tor_j^R(k,T)\cong F_j\otimes k\neq0$ for $0\le j\le g$.
\end{proof}

\begin{lemma}\label{lem:unipotent-twists}
Let $U$ be unipotent of rank $r>0$, $P\in\Pic^0(A)$ and $s\in\Z$.
\begin{enumerate}[label=\textup{(\roman*)}]
\item If $s>0$, then $L^s\otimes P\otimes U$ has cohomology only in degree $0$, of dimension $rs^gd$.
\item If $s<0$, then $L^s\otimes P\otimes U$ has cohomology only in degree $g$, of dimension $r|s|^gd$.
\item If $s=0$ and $P\ne\OO_A$, then $P\otimes U$ is acyclic.
\item If $s=0$ and $P=\OO_A$, then $H^i(A,U)\neq0$ for $0\le i\le g$.
\end{enumerate}
Consequently $L^a\otimes P\otimes U$ is ACM with respect to $L$ if and only if $P\neq\OO_A$.
\end{lemma}

\begin{proof}
Parts (i) and (iii) follow by induction on a unipotent filtration from Lemma~\ref{lem:line-cohomology}; (ii) follows from (i) by Serre duality, since $U^\vee$ is unipotent; (iv) is Lemma~\ref{lem:unipotent-nonvanishing}. Among the twists of $L^a\otimes P\otimes U$ exactly one has exponent zero, which gives the last assertion.
\end{proof}

\subsection{Block decomposition}

\begin{theorem}\label{thm:blocks}
Every nonzero line-filtered bundle $E$ on $A$ admits a decomposition
\begin{equation}\label{eq:blocks}
E\cong\bigoplus_{(a,P)}L^a\otimes P\otimes U_{a,P}
\end{equation}
over finitely many distinct pairs $(a,P)\in\Z\times\Pic^0(A)$, with each $U_{a,P}$ a nonzero unipotent bundle. The summand $E_{a,P}=L^a\otimes P\otimes U_{a,P}$ is unique up to isomorphism, and for every line filtration of $E$ the number of graded pieces isomorphic to $L^a\otimes P$ equals $\rk U_{a,P}$. The following are equivalent:
\begin{enumerate}[label=\textup{(\roman*)}]
\item $E$ is ACM with respect to $L$;
\item every occurring $P$ is nontrivial;
\item every graded piece of some line filtration of $E$ is ACM;
\item every graded piece of every line filtration of $E$ is ACM.
\end{enumerate}
\end{theorem}

\begin{proof}
\emph{Existence.} We induct on the length of a line filtration. Let $0\to F\to E\to L^b\otimes Q\to0$ be its last step and decompose $F$ by induction. For a block of $F$,
\[
\Ext^1\big(L^b\otimes Q,\,L^a\otimes P\otimes U\big)\cong H^1\big(A,L^{a-b}\otimes P\otimes Q^{-1}\otimes U\big),
\]
which vanishes by Lemma~\ref{lem:unipotent-twists} unless $(a,P)=(b,Q)$; here we use $g\ge2$ when $a<b$. Hence the extension class lies in the $(b,Q)$-component and all other blocks split off. The remaining extension, tensored by $L^{-b}\otimes Q^{-1}$, is an extension of $\OO_A$ by a unipotent bundle, hence unipotent. If $F$ has no $(b,Q)$-block, the extension splits and $L^b\otimes Q$ is a new block of rank one. The same induction shows that $\rk U_{a,P}$ is the number of graded pieces of the chosen filtration isomorphic to $L^a\otimes P$.

\emph{Uniqueness.} By Atiyah's Krull--Schmidt theorem \cite{Ati56}, the indecomposable summands of $E$ are unique up to isomorphism and multiplicity. An indecomposable summand of $E_{a,P}$ has the form $L^a\otimes P\otimes V$ with $V$ unipotent and indecomposable. It determines $a$, because its slope is $aH^g$, and it determines $P$, because $\Hom(L^a\otimes P',L^a\otimes P\otimes V)=H^0(A,P\otimes P'^{-1}\otimes V)$ is nonzero if and only if $P'\cong P$ by Lemma~\ref{lem:unipotent-twists}. Hence $E_{a,P}$ is the sum of the indecomposable summands of type $(a,P)$, and is unique. Since the rank of $E_{a,P}$ does not depend on the filtration, neither does the number of graded pieces isomorphic to $L^a\otimes P$.

\emph{Equivalences.} A direct sum is ACM exactly when each summand is, so (i)$\Leftrightarrow$(ii) by Lemma~\ref{lem:unipotent-twists}. By the multiplicity statement, (ii) is equivalent to the condition that, in any given line filtration, no graded piece is of the form $L^a\otimes\OO_A$; by Corollary~\ref{cor:ACM-lines} this is (iii), and also (iv).
\end{proof}

\begin{remark}
Theorem~\ref{thm:blocks} reduces line-filtered bundles to unipotent bundles; it does not classify unipotent bundles, which form a wild class when $g\ge2$ \cite{MB26}. The splitting \eqref{eq:blocks} is unique up to isomorphism but is not canonical as a decomposition of $E$.
\end{remark}

\subsection{Numerically flat bundles}
The next statement does not use the hypothesis on $\NS(A)$.

\begin{proposition}\label{prop:numflat}
Let $A$ be any complex abelian variety of dimension $g\ge2$, $L$ ample, and $E\neq0$ numerically flat. Then $E\cong\bigoplus_{P\in S}P\otimes U_P$ for a finite set $S\subset\Pic^0(A)$ and nonzero unipotent bundles $U_P$. The following are equivalent:
\begin{enumerate}[label=\textup{(\roman*)}]
\item $E$ is ACM with respect to $L$;
\item $E$ is ACM with respect to every ample line bundle;
\item $E$ is acyclic;
\item no slope Jordan--H\"older factor of $E$ is $\OO_A$.
\end{enumerate}
\end{proposition}

\begin{proof}
The decomposition is the classical description of homogeneous bundles \cite[\S3]{Muk81}; it also follows from \cite[Theorem~1.18]{DPS94}, which gives a filtration of $E$ with hermitian flat quotients. Since $\pi_1(A)\cong\Z^{2g}$ is abelian, a unitary representation is a sum of characters, so the quotients are sums of elements of $\Pic^0(A)$. For $P\ne Q$ in $\Pic^0(A)$ and $U$ unipotent, $H^\bullet(A,P\otimes Q^{-1}\otimes U)=0$, and the proof of Theorem~\ref{thm:blocks} restricted to degree-zero lines gives the decomposition. The lines $P$ are stable of slope zero, so they are the Jordan--H\"older factors for every polarization.

Let $B$ be ample. For $t>0$ the line bundles $B^t\otimes P$ are ample and $B^{-t}\otimes P$ antiample, so the twists $P\otimes U_P\otimes B^{t}$, $t\ne0$, have no intermediate cohomology. At $t=0$ the block is acyclic if $P\ne\OO_A$ and has nonzero intermediate cohomology if $P=\OO_A$, by Lemma~\ref{lem:unipotent-nonvanishing}. All four conditions are therefore equivalent to $\OO_A\notin S$.
\end{proof}

\subsection{Slope stability}

\begin{proposition}\label{prop:HN}
Let $E$ be line-filtered with decomposition \eqref{eq:blocks}, and put $B_a=\bigoplus_PL^a\otimes P\otimes U_{a,P}$. Each nonzero $B_a$ is $\mu_L$-semistable of slope $aH^g$, and the partial sums of the $B_a$ in decreasing order of $a$ form the Harder--Narasimhan filtration of $E$. Thus $E$ is semistable if and only if only one value of $a$ occurs, and $E$ is not stable if $\rk E\ge2$.
\end{proposition}

\begin{proof}
Each $B_a$ is an iterated extension of stable line bundles of slope $aH^g$, and an extension of semistable sheaves of equal slope is semistable. The successive quotients of the stated filtration are therefore semistable of strictly decreasing slopes, which characterizes the Harder--Narasimhan filtration. If a single value of $a$ occurs and $\rk E\ge2$, the first piece of a line filtration has slope $\mu_L(E)$.
\end{proof}

\section{Rank-two blocks and Fourier--Mukai parameters}\label{sec:rank-two}

\subsection{The normalized list}

\begin{lemma}\label{lem:self-ext}
Let $0\to\OO_A\xrightarrow{i}U\xrightarrow{q}\OO_A\to0$ be nonsplit. Then $H^0(A,U)=\C\, i$ and $\End(U)\cong\C[\varepsilon]/(\varepsilon^2)$, with $\varepsilon=iq$. In particular $U$ is indecomposable.
\end{lemma}

\begin{proof}
The connecting map $H^0(\OO_A)\to H^1(\OO_A)$ sends $1$ to the extension class, so it is injective and $H^0(U)=\C\, i$. For $\varphi\in\End(U)$ write $\varphi i=ci$; then $\varphi-c\,\id$ vanishes on $i$, hence equals $\psi q$ with $\psi\in H^0(U)=\C\, i$. Thus $\End(U)=\C\,\id\oplus\C\, iq$ with $(iq)^2=0$. This ring is local, so $U$ is indecomposable.
\end{proof}

For a rank-two bundle $E$ define $\nu(E)\in\Z$ by $c_1(E)\equiv\nu(E)H$. Since $\nu(E(t))=\nu(E)+2t$, every rank-two bundle has a unique twist with $\nu\in\{0,1\}$, called its \emph{normalized twist}.

\begin{theorem}\label{thm:rank-two}
A line-filtered rank-two bundle is ACM with respect to $L$ if and only if its normalized twist is one of the following, with $P,Q\in\Pic^0(A)$ nontrivial:
\begin{enumerate}[label=\textup{(\Roman*)}]
\item $(L^k\otimes P)\oplus(L^{-k}\otimes Q)$ with $k\ge0$;
\item a nonsplit extension $0\to P\to E\to P\to0$;
\item $(L^{k+1}\otimes P)\oplus(L^{-k}\otimes Q)$ with $k\ge0$.
\end{enumerate}
Type \textup{(II)} is indecomposable. The normalized semistable bundles in this list are Type \textup{(I)} with $k=0$ and Type \textup{(II)}; they are strictly semistable. All others are unstable.
\end{theorem}

\begin{proof}
By Theorem~\ref{thm:blocks}, a line-filtered rank-two bundle is either a sum of two line bundles or $L^a\otimes P\otimes U$ with $U$ unipotent of rank two, and it is ACM exactly when the degree-zero factors are nontrivial. For $L^{a_1}\otimes P\oplus L^{a_2}\otimes Q$ with $a_1\ge a_2$, normalization means $a_1+a_2\in\{0,1\}$, which gives Types (I) and (III). For a rank-two block $\nu=2a$, so normalization forces $a=0$; the split case is Type (I) with $k=0$, and the nonsplit case is Type (II), indecomposable by Lemma~\ref{lem:self-ext}. The stability statements follow from Proposition~\ref{prop:HN}.
\end{proof}

The parity of $\nu$ separates Types (I) and (III), and indecomposability separates Type (II) from both. For Type (I) with $k=0$ the direct sum is determined by the unordered pair $\{P,Q\}$.

\subsection{Self-extensions and length-two subschemes}
In this subsection parameter spaces are sets of isomorphism classes.

\begin{proposition}\label{prop:self-ext-param}
For $P\in\Pic^0(A)$, the nonsplit self-extensions of $P$, up to isomorphism of the middle term, are in bijection with $\PP(\Ext^1(P,P))=\PP(H^1(A,\OO_A))\cong\PP^{g-1}$. Hence the normalized indecomposable line-filtered rank-two ACM bundles are parametrized by $(\Pic^0(A)\setminus\{\OO_A\})\times\PP^{g-1}$.
\end{proposition}

\begin{proof}
Let $E$ and $E'$ be middle terms of nonzero classes $\xi,\xi'$. By Lemma~\ref{lem:self-ext}, $\Hom(P,E)$ and $\Hom(P,E')$ are one-dimensional, spanned by the distinguished inclusions. An isomorphism $E\to E'$ therefore carries one inclusion to a nonzero multiple of the other and induces scalar automorphisms of subobject and quotient, so $\xi'\in\C^*\xi$. Conversely, proportional classes have isomorphic middle terms. If $P\ne Q$, a self-extension $E_Q$ of $Q$ satisfies $\Hom(P,E_Q)=0$ because $P^{-1}\otimes Q$ is acyclic, whereas $\Hom(P,E)\ne0$; so the families for different $P$ are disjoint.
\end{proof}

\begin{proposition}\label{prop:length-two}
Let $\alpha\in\Ah$ and let $0\to P_\alpha\to E\to P_\alpha\to0$ be nonsplit. Then $\Phi(E)\cong\OO_W[-g]$, where $W\subset\Ah$ is a nonreduced subscheme of length two supported at $-\alpha$. The assignment $E\mapsto W$ is a bijection between $\PP(\Ext^1(P_\alpha,P_\alpha))$ and $\PP(T_{-\alpha}\Ah)$.
\end{proposition}

\begin{proof}
Applying $\Phi$ and Lemma~\ref{lem:FM-twist} gives a triangle $k(-\alpha)[-g]\to\Phi(E)\to k(-\alpha)[-g]$. Its outer terms are sheaves in degree $g$, so $\Phi(E)\cong T[-g]$ with $T$ an extension
\begin{equation}\label{eq:T}
0\to k(-\alpha)\to T\to k(-\alpha)\to0 .
\end{equation}
Since $\Phi$ is an equivalence, it maps $\Ext^1(P_\alpha,P_\alpha)$ isomorphically onto $\Ext^1(k(-\alpha),k(-\alpha))$, and the class of $E$ goes to the class of \eqref{eq:T}, which is therefore nonzero. Let $R=\OO_{\Ah,-\alpha}$ with maximal ideal $\m$. If $\m T=0$, then \eqref{eq:T} is an extension of $k$-vector spaces and splits; hence $\m T\ne0$, and by Nakayama $T/\m T\cong k$. So $T\cong R/I$ is cyclic of length two, with $\m^2\subset I\subsetneq\m$, and $W=\Spec R/I$ is nonreduced.

Ideals $I$ with $\m^2\subset I$ and $\dim_k\m/I=1$ correspond to hyperplanes in $\m/\m^2$, i.e.\ to lines in $T_{-\alpha}\Ah$. The class of $R/I$ in $\Ext^1_R(k,k)=\Hom(\m/\m^2,k)$ is a nonzero functional with kernel $I/\m^2$, determined up to a scalar. Conversely $I=\operatorname{Ann}(T)$, and $\m T$ is the unique submodule of $T$ isomorphic to $k$, so $T$ determines the extension class up to scalar. This gives the bijection.
\end{proof}

\begin{remark}
Under $\Ext^1(P_\alpha,P_\alpha)\cong H^1(A,\OO_A)\cong T_0\Ah$, the bijection of Proposition~\ref{prop:length-two} is translation by $-\alpha$. The split extension corresponds to $k(-\alpha)^{\oplus2}[-g]$. Hence nonsplit self-extensions of nontrivial elements of $\Pic^0(A)$ correspond to nonreduced length-two subschemes of $\Ah$ supported off the origin, which form $\PP(T(\Ah\setminus\{0\}))\cong(\Ah\setminus\{0\})\times\PP^{g-1}$.
\end{remark}

\begin{proposition}\label{prop:S-equiv}
The $S$-equivalence classes of normalized semistable line-filtered rank-two ACM bundles are parametrized by $\Sym^2(\Pic^0(A)\setminus\{\OO_A\})$.
\end{proposition}

\begin{proof}
The Jordan--H\"older graded object of such a bundle is $P\oplus Q$ with $P,Q\ne\OO_A$. Since $\chi(P\otimes L^t)=\chi(L^t)$ for all $P\in\Pic^0(A)$, the factors have the same reduced Hilbert polynomial, so this is also the Gieseker graded object. The unordered pair $\{P,Q\}$ determines it, and every pair occurs.
\end{proof}

\subsection{Curvilinear blocks}
For $P\in\Pic^0(A)$ let $x_P\in\Ah$ be the point with $P_{x_P}\cong P^{-1}$; thus $x_{P_\alpha}=-\alpha$ and, by Lemma~\ref{lem:FM-twist}, $\Phi(P)\cong k(x_P)[-g]$.

\begin{proposition}\label{prop:curvilinear}
Fix $P\in\Pic^0(A)\setminus\{\OO_A\}$ and $r\ge2$. Choose a regular system of parameters $x_1,\dots,x_g$ of $\widehat{\OO}_{\Ah,x_P}$ and put $T_r=\widehat{\OO}_{\Ah,x_P}/(x_1^r,x_2,\dots,x_g)$. There is an indecomposable ACM bundle $E_r$ of rank $r$, filtered by copies of $P$, with $\Phi(E_r)\cong T_r[-g]$. Moreover $\End(E_r)\cong\C[x]/(x^r)$ and
\[
\dim\Ext^i(E_r,E_r)=r\binom gi\qquad(0\le i\le g).
\]
\end{proposition}

The bundles $E_r$ are those constructed in \cite{MB26}; we include the computation of their self-extensions.

\begin{proof}
Finite-length modules over $\OO_{\Ah,x_P}$ and over its completion coincide, so $T_r$ is a sheaf of length $r$ supported at $x_P$, with a composition series whose quotients are $k(x_P)$. Applying $\Phi^{-1}$ to $T_r[-g]$ and inducting on this series gives a rank-$r$ bundle $E_r$ filtered by copies of $\Phi^{-1}(k(x_P)[-g])=P$; it is ACM by Lemma~\ref{lem:ext-ACM}. By full faithfulness $\Ext^i(E_r,E_r)\cong\Ext^i(T_r,T_r)$. The ring $\End(T_r)=T_r\cong\C[x]/(x^r)$ is local, so $E_r$ is indecomposable. The sequence $x_1^r,x_2,\dots,x_g$ is regular; its Koszul complex resolves $T_r$, and all differentials of $\Hom(K_\bullet,T_r)$ are multiplications by elements of $\operatorname{Ann}(T_r)$, hence zero. So $\Ext^i(T_r,T_r)$ is a direct sum of $\binom gi$ copies of $T_r$.
\end{proof}

\begin{remark}
For a nonsplit rank-two block $E$ one gets $\dim\End(E)=2$ and $\dim\Ext^1(E,E)=2g$. Since $\Ext^2(E,E)\neq0$, these dimensions alone do not show that the deformations of $E$ are unobstructed.
\end{remark}

\section{Saturated line subbundles and defects}\label{sec:defects}

Lemma~\ref{lem:saturation} and Theorem~\ref{thm:defect-local} hold on any smooth projective variety $X$; Corollaries~\ref{cor:omega-Z} and~\ref{cor:stable-defect} use the abelian variety $A$ with $\NS(A)=\Z[L]$. A subsheaf $N\subset E$ is \emph{saturated} if $E/N$ is torsion-free.

\begin{lemma}\label{lem:saturation}
Every rank-two vector bundle $E$ on a smooth projective variety $X$ admits an exact sequence $0\to N\to E\to M\otimes\I_Z\to0$ with $N$, $M$ line bundles and $Z$ empty or of codimension at least two. One may take $N$ to be the saturation of any rank-one subsheaf and $M=(E/N)^{\vee\vee}$.
\end{lemma}

\begin{proof}
For $B$ ample and $k\gg0$, a nonzero section of $E\otimes B^k$ gives an injection $B^{-k}\to E$. Let $N$ be the saturation of a rank-one subsheaf and $G=E/N$. At a point $x$ with $\dim\OO_{X,x}\ge1$, torsion-freeness gives $\depth G_x\ge1$ when $G_x\ne0$, and the depth lemma gives $\depth N_x\ge\min\{2,\dim\OO_{X,x}\}$. So $N$ is reflexive; a rank-one reflexive sheaf on a smooth variety is invertible \cite[Proposition~1.9]{Har80}. Likewise $M=G^{\vee\vee}$ is a line bundle, and $G\hookrightarrow M$ is an isomorphism in codimension one, so $G\otimes M^{-1}=\I_Z$ with $\codim Z\ge2$.
\end{proof}

We call $Z$ the \emph{defect} of the saturated subsheaf $N\subset E$. It depends on $N$, not only on $E$.

\begin{theorem}\label{thm:defect-local}
Let $X$ be smooth and $0\to N\to E\to M\otimes\I_Z\to0$ exact, with $E$ of rank two, $N,M$ line bundles, and $Z$ empty or of codimension at least two. Then $M\cong\det E\otimes N^{-1}$, and $Z$ is the zero scheme of the induced section $s\in H^0(X,E\otimes N^{-1})$. If $Z\ne\emptyset$, then $s$ is a regular section, $Z$ is a local complete intersection of pure codimension two, and
\[
[Z]=c_2(E\otimes N^{-1}),\qquad \omega_Z\cong(\omega_X\otimes\det E\otimes N^{-2})|_Z .
\]
In all cases there is an exact Koszul complex
\begin{equation}\label{eq:koszul}
0\to\det E^{-1}\otimes N^2\to E^\vee\otimes N\to\OO_X\to\OO_Z\to0 .
\end{equation}
\end{theorem}

\begin{proof}
Let $s\in H^0(X,E\otimes N^{-1})$ be the section defined by $N\subset E$, and let $W$ be its zero scheme. Locally write $E=\OO e_1\oplus\OO e_2$ and $s=fe_1+he_2$, so that $\I_W=(f,h)$.

\emph{Step 1.} No height-one prime $\mathfrak p$ contains both $f$ and $h$: otherwise in the discrete valuation ring $\OO_{X,\mathfrak p}$ we could write $s=\pi^k(f',h')$ with $k\ge1$, and the class of $(f',h')$ would be a nonzero torsion element of $(E/N)_{\mathfrak p}$. Hence $W$ is empty or of codimension at least two, and at a point where $(f,h)$ is a proper ideal it has height two by Krull's theorem. Since $\OO_{X,x}$ is Cohen--Macaulay, $f,h$ is then a regular sequence.

\emph{Step 2.} Put $M'=\det E\otimes N^{-1}$ and let $\psi\colon E\to M'$ be $e\mapsto s\wedge e$. Locally $\psi(e_1)=-h\,e_1\wedge e_2$ and $\psi(e_2)=f\,e_1\wedge e_2$, so $\operatorname{im}\psi=M'\otimes\I_W$. The kernel of $\psi$ is $N$: where $(f,h)$ is proper this is exactness of the Koszul complex of the regular sequence $f,h$, and elsewhere the row $(-h,f)$ is unimodular. Hence $\psi$ induces an isomorphism $\bar\psi\colon E/N\to M'\otimes\I_W$, and the dual Koszul complex of $s$ is \eqref{eq:koszul} with $W$ in place of $Z$.

\emph{Step 3.} Let $\bar q\colon E/N\to M$ be the injection with image $M\otimes\I_Z$ given by the hypothesis. On $U=X\setminus(Z\cup W)$, whose complement has codimension at least two, $\theta_U=\bar\psi\circ\bar q^{-1}\colon M|_U\to M'|_U$ is an isomorphism of line bundles. Since $X$ is normal, $\theta_U$ and $\theta_U^{-1}$ extend to mutually inverse homomorphisms, so $\theta_U$ extends to an isomorphism $\theta\colon M\to M'$. The maps $\theta\circ\bar q$ and $\bar\psi$ agree on $U$; their difference has image in the torsion-free sheaf $M'$ supported on $X\setminus U$, hence vanishes. Thus $\theta(M\otimes\I_Z)=M'\otimes\I_W$, and since $\theta$ is locally multiplication by a unit, $\I_Z=\I_W$. In particular $M\cong\det E\otimes N^{-1}$ and $Z=W$.

\emph{Step 4.} If $Z\ne\emptyset$, Step 1 shows that $s$ is regular, so $Z$ is a local complete intersection of pure codimension two with normal bundle $(E\otimes N^{-1})|_Z$. Adjunction gives the formula for $\omega_Z$, and $[Z]=c_2(E\otimes N^{-1})$ by \cite[Example~3.2.16]{Ful98}.
\end{proof}

\begin{corollary}\label{cor:omega-Z}
On $A$, if $N=L^a\otimes P$ and $M=L^b\otimes Q$, a nonempty defect satisfies $\omega_Z\cong(L^{b-a}\otimes Q\otimes P^{-1})|_Z$ and $[Z]\equiv c_2(E)-abH^2$.
\end{corollary}

\begin{proof}
Use $\omega_A\cong\OO_A$, $\det E\cong L^{a+b}\otimes P\otimes Q$, and $c_2(E\otimes N^{-1})=c_2(E)-c_1(E)c_1(N)+c_1(N)^2=c_2(E)-c_1(N)c_1(M)$.
\end{proof}

\begin{corollary}\label{cor:stable-defect}
If $E$ is a $\mu_L$-stable rank-two bundle on $A$, every saturated line subbundle of $E$ has nonempty defect of pure dimension $g-2$.
\end{corollary}

\begin{proof}
An empty defect makes $E$ line-filtered, which is excluded by Proposition~\ref{prop:HN}.
\end{proof}

\section{Ulrich bundles}\label{sec:ulrich}

\begin{definition}
A vector bundle $E$ on $(A,L)$ is \emph{Ulrich} if $H^i(A,E(-p))=0$ for all $i$ and all $1\le p\le g$.
\end{definition}

For $L$ very ample this is the cohomological characterization of Ulrich bundles in \cite[Proposition~2.1]{ES03}. All arguments below use only these vanishings, so they apply to any ample $L$.

\subsection{The line-filtered obstruction}

\begin{corollary}\label{cor:no-ulrich}
No nonzero line-filtered bundle on $A$ is Ulrich.
\end{corollary}

\begin{proof}
By Theorem~\ref{thm:blocks} it suffices to treat a block $L^a\otimes P\otimes U$. By Lemma~\ref{lem:unipotent-twists}, its twist by $L^{-p}$ can be acyclic only if $a=p$. Acyclicity for $p=1$ and $p=2$ is impossible.
\end{proof}

\subsection{Numerical identities}

\begin{lemma}\label{lem:hilbert-ulrich}
If $E$ is Ulrich of rank $r$, then $\chi(A,E(t))=r\,d\prod_{j=1}^g(t+j)$.
\end{lemma}

\begin{proof}
$\chi(E(t))$ is a polynomial of degree $g$ with leading coefficient $rH^g/g!=rd$ that vanishes at $t=-1,\dots,-g$.
\end{proof}

\begin{proposition}\label{prop:chern}
A rank-two Ulrich bundle $E$ on $(A,L)$ satisfies
\begin{gather*}
c_1(E)\equiv(g+1)H,\\
\ch_2(E)\cdot H^{g-2}=\frac{(g+1)(3g+2)}{12}H^g,\qquad c_2(E)\cdot H^{g-2}=\frac{(g+1)(3g+4)}{12}H^g .
\end{gather*}
\end{proposition}

\begin{proof}
Since $\operatorname{td}(A)=1$, $\chi(E(t))=\sum_k\ch_{g-k}(E)\cdot H^k\,t^k/k!$. The coefficients of $t^{g-1}$ and $t^{g-2}$ are $c_1(E)\cdot H^{g-1}/(g-1)!$ and $\ch_2(E)\cdot H^{g-2}/(g-2)!$. By Lemma~\ref{lem:hilbert-ulrich} they equal $2d\,e_1$ and $2d\,e_2$, where $e_k$ is the $k$-th elementary symmetric function of $1,\dots,g$:
\[
e_1=\frac{g(g+1)}2,\qquad e_2=\frac{(g-1)g(g+1)(3g+2)}{24}.
\]
This gives $c_1(E)\cdot H^{g-1}=(g+1)H^g$ and the value of $\ch_2(E)\cdot H^{g-2}$. As $\NS(A)=\Z[L]$, $c_1(E)\equiv(g+1)H$. Finally $c_2=c_1^2/2-\ch_2$.
\end{proof}

For instance, when $g=2$ these give $c_1(E)\equiv3H$ and $c_2(E)=5d$; when $g=3$, $c_2(E)\cdot H=26d$, and Riemann--Roch gives $\chi(E)=c_1^3/6-c_1c_2/2=64d-52d=12d$, in agreement with Lemma~\ref{lem:hilbert-ulrich}.

\subsection{The defect bound}

\begin{theorem}\label{thm:defect-bound}
Let $E$ be a rank-two Ulrich bundle on $(A,L)$ and let
\begin{equation}\label{eq:saturated}
0\to L^a\otimes P\to E\to L^b\otimes Q\otimes\I_Z\to0,\qquad P,Q\in\Pic^0(A),
\end{equation}
be a saturated sequence. Then $a+b=g+1$, $a\le1$, and
\begin{equation}\label{eq:degree}
[Z]\cdot H^{g-2}=\Big(\frac{(g+1)(3g+4)}{12}-ab\Big)H^g\;\ge\;\frac{3g^2-5g+4}{12}\,H^g>0 .
\end{equation}
Equality holds if and only if $a=1$, i.e.\ $(a,b)=(1,g)$, and then $P\ne\OO_A$ and $Q\ne\OO_A$. The defect is a nonempty local complete intersection of pure codimension two with $\omega_Z\cong(L^{b-a}\otimes Q\otimes P^{-1})|_Z$.
\end{theorem}

\begin{proof}
Taking determinants and using Proposition~\ref{prop:chern} gives $a+b=g+1$. Corollary~\ref{cor:omega-Z} and Proposition~\ref{prop:chern} give the identity in \eqref{eq:degree}.

Twisting the inclusion by $L^{-1}$ gives $H^0(A,L^{a-1}\otimes P)\subset H^0(A,E(-1))=0$. For $a\ge2$ the left side has dimension $(a-1)^gd>0$, so $a\le1$; for $a=1$ it is $H^0(A,P)$, so $P\ne\OO_A$.

Assume $a=1$. By Serre duality and the vanishing $H^\bullet(E(-g))=0$, the bundle $E^\vee\otimes L^g$ is acyclic. Since $E^\vee\cong E\otimes\det E^{-1}$ and $\det E\cong L^{g+1}\otimes P\otimes Q$, we have $E^\vee\otimes L^g\cong E\otimes L^{-1}\otimes P^{-1}\otimes Q^{-1}$. Tensoring \eqref{eq:saturated} by $L^{-1}\otimes P^{-1}\otimes Q^{-1}$ shows that this bundle contains $L^{a-1}\otimes Q^{-1}=Q^{-1}$, so $H^0(A,Q^{-1})=0$ and $Q\ne\OO_A$.

With $b=g+1-a$, the function $a\mapsto a(g+1-a)$ is increasing for $a<(g+1)/2$, so $a\le1$ gives $ab\le g$, with equality only for $a=1$. Since $\frac{(g+1)(3g+4)}{12}-g=\frac{3g^2-5g+4}{12}$ and $3g^2-5g+4$ has negative discriminant, the bound is positive and $Z\ne\emptyset$. The remaining statements follow from Theorem~\ref{thm:defect-local} and Corollary~\ref{cor:omega-Z}.
\end{proof}

\begin{corollary}\label{cor:stable}
Every rank-two Ulrich bundle on $(A,L)$ is $\mu_L$-stable.
\end{corollary}

\begin{proof}
By Lemma~\ref{lem:saturation} it suffices to test saturated line subbundles $L^a\otimes P$. By Theorem~\ref{thm:defect-bound}, $\mu_L(L^a\otimes P)=aH^g\le H^g<\frac{g+1}2H^g=\mu_L(E)$.
\end{proof}

\begin{remark}
Corollary~\ref{cor:stable} can also be deduced from \cite[Theorem~2.9]{CH12}: an Ulrich bundle is $\mu$-semistable, and a saturated subsheaf of the same slope is again Ulrich, while $A$ carries no Ulrich line bundle by Corollary~\ref{cor:no-ulrich}. The proof above gives in addition the bound $a\le1$, which is what enters \eqref{eq:degree}. Using only $a+b=g+1$ one would get $ab\le(g+1)^2/4$; this coincides with $ab\le g$ for $g=2$ and is weaker for $g\ge3$.
\end{remark}

\subsection{The Hilbert polynomial of the defect}

\begin{proposition}\label{prop:hilbert}
In the situation of Theorem~\ref{thm:defect-bound},
\[
\chi(Z,\OO_Z\otimes L^t)=d\Big[t^g-2\prod_{k=1}^{g}(t+a-k)+(t+2a-g-1)^g\Big].
\]
In particular, in the equality case $a=1$,
\[
\chi(Z,\OO_Z\otimes L^t)=d\Big[t^g+(t-g+1)^g-2\,t(t-1)\cdots(t-g+1)\Big],\qquad \chi(Z,\OO_Z)=(1-g)^g\,d .
\]
\end{proposition}

\begin{proof}
Twist the Koszul complex \eqref{eq:koszul} with $N=L^a\otimes P$ by $L^t$. Since $E^\vee\otimes N\cong E\otimes L^{a-g-1}\otimes Q^{-1}$ and $\det E^{-1}\otimes N^2\cong L^{2a-g-1}\otimes P\otimes Q^{-1}$, additivity of $\chi$ and the invariance of $\chi$ under twists by $\Pic^0(A)$ give
\[
\chi(\OO_Z(t))=\chi(L^t)-\chi\big(E(t+a-g-1)\big)+\chi(L^{t+2a-g-1}).
\]
Lemma~\ref{lem:line-cohomology} gives the outer terms, and Lemma~\ref{lem:hilbert-ulrich} gives $\chi(E(t+a-g-1))=2d\prod_{j=1}^g(t+a-g-1+j)=2d\prod_{k=1}^g(t+a-k)$.
\end{proof}

The coefficients of $t^g$ and $t^{g-1}$ in Proposition~\ref{prop:hilbert} vanish, as they must, and the coefficient of $t^{g-2}$ is $[Z]\cdot H^{g-2}/(g-2)!$, recovering \eqref{eq:degree}. For $g=2$ the polynomial is the constant $(2a^2-6a+5)d=\length(Z)$.

\begin{corollary}\label{cor:threefold}
Let $g=3$ and suppose a rank-two Ulrich bundle has a saturated subsheaf with $a=1$. Then the defect $Z$ is a locally complete intersection curve with $\deg_L Z=8d$, arithmetic genus $p_a(Z)=8d+1$, and $\omega_Z\cong(L^2\otimes Q\otimes P^{-1})|_Z$.
\end{corollary}

\begin{proof}
Proposition~\ref{prop:hilbert} gives $\chi(\OO_Z(t))=d(8t-8)$, so $\deg_LZ=8d$ and $\chi(\OO_Z)=-8d$. The formula for $\omega_Z$ is Theorem~\ref{thm:defect-bound}; as a check, $\deg\omega_Z=2\deg_LZ=16d=2p_a(Z)-2$.
\end{proof}

\subsection{Minimal defects and the jump locus}

\begin{lemma}\label{lem:top-vanishing}
Let $E$ be a rank-two Ulrich bundle on $(A,L)$. Then $H^0(A,E\otimes L^{-k}\otimes R)=0$ for every $k\ge2$ and every $R\in\Pic^0(A)$, and $H^g(A,E\otimes L^{-1}\otimes R)=0$ for every $R\in\Pic^0(A)$. In particular $V^g(E\otimes L^{-1})=\emptyset$.
\end{lemma}

\begin{proof}
A nonzero section of $E\otimes L^{-k}\otimes R$ gives an injection $L^k\otimes R^{-1}\to E$. Its saturation is a line bundle $L^a\otimes P'$ with $H^0(A,L^{a-k}\otimes P'\otimes R)\ne0$, so $a\ge k\ge2$ by Lemma~\ref{lem:line-cohomology}, contradicting Theorem~\ref{thm:defect-bound}. Write $\det E\cong L^{g+1}\otimes\varepsilon$ with $\varepsilon\in\Pic^0(A)$, so that $E^\vee\cong E\otimes L^{-g-1}\otimes\varepsilon^{-1}$. Serre duality gives $H^g(A,E\otimes L^{-1}\otimes R)\cong H^0(A,E\otimes L^{-g}\otimes\varepsilon^{-1}R^{-1})^\vee$, which vanishes by the first part because $g\ge2$.
\end{proof}

\begin{theorem}\label{thm:jump}
Let $E$ be a rank-two Ulrich bundle on $(A,L)$ and put $V=V^0(E\otimes L^{-1})=\{R\in\Pic^0(A):H^0(A,E\otimes L^{-1}\otimes R)\neq0\}$.
\begin{enumerate}[label=\textup{(\roman*)}]
\item $V$ is a closed subset of $\Pic^0(A)$ not containing $\OO_A$.
\item For $R\in V$, every nonzero map $L\otimes R^{-1}\to E$ is a saturated embedding, with defect $Z$ satisfying equality in \eqref{eq:degree}. Conversely, every saturated subsheaf with $a=1$ is $L\otimes R^{-1}$ for some $R\in V$.
\item If $V=\emptyset$, then every saturated line subbundle of $E$ has $a\le0$, every defect satisfies
\[
[Z]\cdot H^{g-2}\ge\frac{(g+1)(3g+4)}{12}H^g=c_2(E)\cdot H^{g-2},
\]
with equality exactly when $a=0$, and $V^i(E\otimes L^{-1})\neq\emptyset$ for some $1\le i\le g-1$.
\end{enumerate}
\end{theorem}

\begin{proof}
(i) $V$ is closed by semicontinuity, and $\OO_A\notin V$ because $H^0(E(-1))=0$.

(ii) Let $R\in V$ and $\varphi\colon L\otimes R^{-1}\to E$ nonzero. Since $E$ is torsion-free, $\varphi$ is injective. Let $N=L^a\otimes P'$ be the saturation of its image. Then $L\otimes R^{-1}\subset N$, so $H^0(A,L^{a-1}\otimes P'\otimes R)\neq0$. By Lemma~\ref{lem:line-cohomology} this forces $a\ge1$, and by Theorem~\ref{thm:defect-bound} $a\le1$. Hence $a=1$, $P'\cong R^{-1}$, and the inclusion $L\otimes R^{-1}\subset N$ is a nonzero endomorphism of a line bundle on a complete variety, hence an isomorphism. So $\varphi$ is saturated with $a=1$. Conversely, a saturated subsheaf $L\otimes P\subset E$ gives a nonzero section of $E\otimes L^{-1}\otimes P^{-1}$, so $P^{-1}\in V$.

(iii) If $V=\emptyset$, part (ii) shows that no saturated subsheaf has $a=1$, so $a\le0$ by Theorem~\ref{thm:defect-bound}. Then $b=g+1-a>0$ and $ab\le0$, with equality iff $a=0$, which gives the inequality from \eqref{eq:degree}.

By Lemma~\ref{lem:top-vanishing}, $V^g(E\otimes L^{-1})=\emptyset$. If also $V^i(E\otimes L^{-1})=\emptyset$ for $1\le i\le g-1$, then $E\otimes L^{-1}\otimes R$ would be acyclic for every $R\in\Pic^0(A)$, and Lemma~\ref{lem:fibres} would give $E=0$.
\end{proof}

\begin{remark}
Theorem~\ref{thm:jump} reduces the existence of minimal-defect subsheaves to a generic vanishing question for the single sheaf $E\otimes L^{-1}$, whose Euler characteristic $\chi(E\otimes L^{-1}\otimes R)$ is zero for every $R$. The higher-dimensional equality case of Theorem~\ref{thm:defect-bound} occurs for $E$ if and only if $V^0(E\otimes L^{-1})\ne\emptyset$.
\end{remark}

\subsection{Minimal-defect extensions in dimension at least three}

\begin{theorem}\label{thm:higher-criterion}
Let $g\ge3$, let $P,Q\in\Pic^0(A)$, and let $Z\subset A$ be empty or a closed subscheme of codimension at least two. Suppose that
\begin{equation}\label{eq:higher-ext}
0\to L\otimes P\to E\to L^g\otimes Q\otimes\I_Z\to0
\end{equation}
is exact with $E$ locally free. Then $E$ is Ulrich if and only if $P\ne\OO_A$, $Q\ne\OO_A$, and
\begin{equation}\label{eq:higher-vanishing}
H^i\big(A,L^k\otimes R\otimes\I_Z\big)=0\qquad\text{for } R\in\{Q,P^{-1}\},\ 0\le k\le g-1,\ 0\le i\le g-2 .
\end{equation}
If $P,Q\ne\OO_A$, condition \eqref{eq:higher-vanishing} holds if and only if, for $R\in\{Q,P^{-1}\}$:
\begin{enumerate}[label=\textup{(\alph*)}]
\item $H^j(Z,R|_Z)=0$ for $0\le j\le g-3$;
\item for $1\le k\le g-1$, the restriction map $H^0(A,L^k\otimes R)\to H^0(Z,L^k\otimes R|_Z)$ is an isomorphism and $H^j(Z,L^k\otimes R|_Z)=0$ for $1\le j\le g-3$.
\end{enumerate}
\end{theorem}

\begin{proof}
The sequence \eqref{eq:higher-ext} is saturated, so if $E$ is Ulrich then $P,Q\ne\OO_A$ by Theorem~\ref{thm:defect-bound}. We assume $P,Q\ne\OO_A$ from now on. Since $\det E\cong L^{g+1}\otimes P\otimes Q$ and $E^\vee\cong E\otimes\det E^{-1}$, twisting \eqref{eq:higher-ext} by $L^{-p}$ and by $\det E^{-1}\otimes L^p$ gives, for $1\le p\le g$,
\begin{gather*}
0\to L^{1-p}\otimes P\to E(-p)\to L^{g-p}\otimes Q\otimes\I_Z\to0,\\
0\to L^{p-g}\otimes Q^{-1}\to E^\vee(p)\to L^{p-1}\otimes P^{-1}\otimes\I_Z\to0 .
\end{gather*}
The line bundle $L^{1-p}\otimes P$ has no cohomology in degrees at most $g-1$: for $p\ge2$ by Lemma~\ref{lem:line-cohomology}(iv), and for $p=1$ because $P\ne\OO_A$. Hence $H^i(E(-p))\cong H^i(L^{g-p}\otimes Q\otimes\I_Z)$ for $0\le i\le g-2$. In the same way, using $Q\ne\OO_A$ when $p=g$, we get $H^i(E^\vee(p))\cong H^i(L^{p-1}\otimes P^{-1}\otimes\I_Z)$ for $0\le i\le g-2$. Serre duality gives $H^j(E(-p))\cong H^{g-j}(E^\vee(p))^\vee$. Since $g\ge3$, every $j\in\{0,\dots,g\}$ satisfies $j\le g-2$ or $g-j\le g-2$. Therefore $E(-p)$ is acyclic if and only if $H^i(L^{g-p}\otimes Q\otimes\I_Z)=0$ and $H^i(L^{p-1}\otimes P^{-1}\otimes\I_Z)=0$ for $0\le i\le g-2$. As $p$ runs from $1$ to $g$, both $g-p$ and $p-1$ run over $0,\dots,g-1$, which proves the first assertion.

For the second, let $R\in\{Q,P^{-1}\}$, so $R\ne\OO_A$, and consider $0\to L^k\otimes R\otimes\I_Z\to L^k\otimes R\to L^k\otimes R|_Z\to0$. If $k\ge1$, then $L^k\otimes R$ has cohomology only in degree zero, so $H^0$ and $H^1$ of $L^k\otimes R\otimes\I_Z$ are the kernel and cokernel of the restriction map, and $H^i(L^k\otimes R\otimes\I_Z)\cong H^{i-1}(Z,L^k\otimes R|_Z)$ for $i\ge2$. If $k=0$, then $R$ is acyclic and $H^i(R\otimes\I_Z)\cong H^{i-1}(Z,R|_Z)$ for all $i$. Since $g-2\ge1$, the vanishing in \eqref{eq:higher-vanishing} for $0\le i\le g-2$ is \textup{(a)} when $k=0$ and \textup{(b)} when $k\ge1$.
\end{proof}

\begin{corollary}\label{cor:HS}
Let $g\ge3$. The following are equivalent:
\begin{enumerate}[label=\textup{(\roman*)}]
\item there is a rank-two Ulrich bundle $E$ on $(A,L)$ with $V^0(E\otimes L^{-1})\ne\emptyset$;
\item there are nontrivial $P,Q\in\Pic^0(A)$ and a nonempty local complete intersection $Z\subset A$ of pure codimension two with $\omega_Z\cong(L^{g-1}\otimes Q\otimes P^{-1})|_Z$ satisfying conditions \textup{(a)} and \textup{(b)} of Theorem~\ref{thm:higher-criterion}.
\end{enumerate}
\end{corollary}

\begin{proof}
(i)$\Rightarrow$(ii). By Theorem~\ref{thm:jump}(ii), $E$ has a saturated subsheaf $L\otimes P$ with $P^{-1}\in V^0(E\otimes L^{-1})$. By Theorem~\ref{thm:defect-bound} the corresponding sequence has the form \eqref{eq:higher-ext} with $P,Q\ne\OO_A$, and $Z$ is a nonempty local complete intersection of pure codimension two with the stated dualizing sheaf. Theorem~\ref{thm:higher-criterion} gives (a) and (b).

(ii)$\Rightarrow$(i). Put $M=L^{g-1}\otimes Q\otimes P^{-1}$. Since $\omega_A\cong\OO_A$, we have $\det N_{Z/A}\cong\omega_Z\cong M|_Z$, and $H^2(A,M^{-1})=0$ by Lemma~\ref{lem:line-cohomology}(iv) because $g\ge3$. By the Hartshorne--Serre correspondence \cite[Theorem~1.1]{Arr07} there is a rank-two vector bundle $F$ with $\det F\cong M$ and a section whose zero scheme is $Z$, that is, an exact sequence $0\to\OO_A\to F\to M\otimes\I_Z\to0$. The bundle $E=F\otimes L\otimes P$ fits into \eqref{eq:higher-ext}, so it is Ulrich by Theorem~\ref{thm:higher-criterion}. The inclusion $L\otimes P\subset E$ gives $P^{-1}\in V^0(E\otimes L^{-1})$.
\end{proof}

\begin{corollary}\label{cor:g3}
Let $g=3$, let $P,Q\in\Pic^0(A)$ be nontrivial, let $Z\subset A$ be a local complete intersection curve, and let $E$ be locally free in \eqref{eq:higher-ext}. Then $E$ is Ulrich if and only if
\begin{enumerate}[label=\textup{(\arabic*)}]
\item $H^0(Z,Q|_Z)=0$ and $H^0(Z,P^{-1}|_Z)=0$, and
\item the restriction maps $H^0(A,L^k\otimes R)\to H^0(Z,L^k\otimes R|_Z)$ are isomorphisms for $k\in\{1,2\}$ and $R\in\{Q,P^{-1}\}$.
\end{enumerate}
In this case $\deg_LZ=8d$, $p_a(Z)=8d+1$, $\omega_Z\cong(L^2\otimes Q\otimes P^{-1})|_Z$, and the four spaces $H^0(Z,L^k\otimes R|_Z)$ have dimensions $d$, $8d$, $d$, $8d$.
\end{corollary}

\begin{proof}
For $g=3$ the range $1\le j\le g-3$ in Theorem~\ref{thm:higher-criterion}(b) is empty, which gives the criterion. The numerical statements follow from Corollary~\ref{cor:threefold} and Lemma~\ref{lem:line-cohomology}(iii).
\end{proof}

\subsection{Abelian surfaces}
In this subsection $g=2$ and $L^2=2d$. Theorem~\ref{thm:defect-bound} reads $\length(Z)=(5-2ab)d\ge d$, with equality exactly for $(a,b)=(1,2)$.

\begin{theorem}\label{thm:surface-criterion}
Let $A$ be an abelian surface and $L$ an ample line bundle with $L^2=2d$; no assumption on $\NS(A)$ is needed. Let $Z\subset A$ be zero-dimensional of length $d$, $P,Q\in\Pic^0(A)$, and let
\begin{equation}\label{eq:surface-ext}
0\to L\otimes P\to E\to L^2\otimes Q\otimes\I_Z\to0
\end{equation}
be exact with $E$ locally free. Then $E$ is Ulrich if and only if $P\ne\OO_A$, $Q\ne\OO_A$, and the restriction maps
\[
r_Q\colon H^0(A,L\otimes Q)\to H^0(\OO_Z),\qquad r_{P}\colon H^0(A,L\otimes P^{-1})\to H^0(\OO_Z)
\]
are isomorphisms.
\end{theorem}

\begin{proof}
For $R\in\Pic^0(A)$ the bundle $L\otimes R$ has no higher cohomology and $h^0(L\otimes R)=d=h^0(\OO_Z)$. From $0\to L\otimes R\otimes\I_Z\to L\otimes R\to L\otimes R|_Z\to0$ we see that $L\otimes R\otimes\I_Z$ is acyclic if and only if $H^0(L\otimes R)\to H^0(\OO_Z)$ is an isomorphism.

Since $\det E\cong L^3\otimes P\otimes Q$, we have $E^\vee\otimes L^2\cong E\otimes L^{-1}\otimes P^{-1}\otimes Q^{-1}$. Twisting \eqref{eq:surface-ext} by $L^{-1}$ and by $L^{-1}\otimes P^{-1}\otimes Q^{-1}$ gives
\[
0\to P\to E(-1)\to L\otimes Q\otimes\I_Z\to0,\qquad 0\to Q^{-1}\to E^\vee\otimes L^2\to L\otimes P^{-1}\otimes\I_Z\to0 .
\]
By Serre duality, $E$ is Ulrich if and only if $E(-1)$ and $E^\vee\otimes L^2$ are acyclic.

If $E$ is Ulrich, then $H^0(P)\subset H^0(E(-1))=0$ and $H^0(Q^{-1})\subset H^0(E^\vee\otimes L^2)=0$, so $P,Q\ne\OO_A$. Then the subobjects in both sequences are acyclic, so the quotients are acyclic, and $r_Q$, $r_P$ are isomorphisms. The converse follows from the same two sequences.
\end{proof}

\begin{remark}
In Theorem~\ref{thm:surface-criterion} local freeness of $E$ is a hypothesis. By the Serre correspondence \cite[Theorem~5.1.1]{HL10}, applied to $E\otimes L^{-1}\otimes P^{-1}$, a locally free extension \eqref{eq:surface-ext} exists if and only if $Z$ is a local complete intersection satisfying the Cayley--Bacharach property with respect to $L\otimes Q\otimes P^{-1}$.
\end{remark}

\begin{theorem}\label{thm:surface-classification}
Let $A$ be an abelian surface with $\NS(A)=\Z[L]$, $L$ ample, $L^2=2d$, and let $E$ be a rank-two vector bundle on $A$. Then $E$ is Ulrich if and only if it fits into an exact sequence \eqref{eq:surface-ext} with $P,Q\ne\OO_A$, $\length(Z)=d$, and $r_Q$, $r_P$ isomorphisms. Moreover, for an Ulrich bundle $E$:
\begin{enumerate}[label=\textup{(\roman*)}]
\item the line bundles $L\otimes P$ occurring in such sequences are exactly the $L\otimes R^{-1}$ with $R\in V^0(E\otimes L^{-1})$, and this set is a nonempty proper closed subset of $\Pic^0(A)$;
\item every other saturated line subbundle $L^a\otimes P$ has $a\le0$ and defect of length $(2a^2-6a+5)d\ge5d$.
\end{enumerate}
\end{theorem}

\begin{proof}
The ``if'' direction is Theorem~\ref{thm:surface-criterion}. Conversely, let $E$ be Ulrich and put $V=V^0(E\otimes L^{-1})$. Suppose $V=\emptyset$. By Lemma~\ref{lem:top-vanishing} and Theorem~\ref{thm:jump}(iii), $V^2(E\otimes L^{-1})=\emptyset$ and $V^1(E\otimes L^{-1})\neq\emptyset$. But for every $R\in\Pic^0(A)$,
\[
h^1(E\otimes L^{-1}\otimes R)=h^0(E\otimes L^{-1}\otimes R)+h^2(E\otimes L^{-1}\otimes R)-\chi(E\otimes L^{-1})=0,
\]
since $\chi(E(-1))=0$. This contradiction shows $V\ne\emptyset$. By Theorem~\ref{thm:jump}(ii), for $R\in V$ there is a saturated sequence with $N=L\otimes R^{-1}$ and $a=1$; Theorem~\ref{thm:defect-bound} gives $b=2$, $\length(Z)=d$ and $P,Q\neq\OO_A$, and Theorem~\ref{thm:surface-criterion} gives the restriction isomorphisms. Part (i) is Theorem~\ref{thm:jump}(i)--(ii) together with $V\ne\emptyset$. For (ii), $a\le1$ and $a\ne1$ give $a\le0$, and $\length(Z)=(5-2a(3-a))d=(2a^2-6a+5)d\ge5d$.
\end{proof}

\begin{proposition}\label{prop:divisor}
Let $A$ be an abelian surface with $\NS(A)=\Z[L]$, $L$ ample, $L^2=2d$, and let $E$ be a rank-two Ulrich bundle on $(A,L)$. Then $V^0(E\otimes L^{-1})=V^1(E\otimes L^{-1})$, and this set is the support of an ample effective divisor $D_E$ on $\Ah=\Pic^0(A)$ with $D_E^2=2d$. In particular $V^0(E\otimes L^{-1})$ is a connected curve, and the minimal-defect subsheaves of $E$ form a one-dimensional family.
\end{proposition}

\begin{proof}
Put $F=E\otimes L^{-1}$ and $\mathcal F=p_A^*F\otimes\Pcal$ on $A\times\Ah$, so that $\mathcal F|_{A\times\{\alpha\}}\cong F\otimes P_\alpha$ and $\R p_{\Ah*}\mathcal F=\Phi(F)$. By Lemma~\ref{lem:top-vanishing}, $H^2(A,F\otimes P_\alpha)=0$ for every $\alpha$, and $\chi(F\otimes P_\alpha)=\chi(E(-1))=0$ by Lemma~\ref{lem:hilbert-ulrich}. Hence $h^0(F\otimes P_\alpha)=h^1(F\otimes P_\alpha)$ and $V^0(F)=V^1(F)$.

By the theorem on cohomology and base change \cite[\S5]{Mum70}, every point of $\Ah$ has an affine neighbourhood $\Spec B$ over which $\Phi(F)$ is represented by a complex $K^0\to K^1\to K^2$ of finite free $B$-modules computing the cohomology of every fibre. Since $H^2$ of every fibre vanishes, $K^1\to K^2$ is surjective modulo every maximal ideal, hence surjective, and its kernel $K'$ is projective. Thus over $\Spec B$ the object $\Phi(F)$ is represented by a map $\varphi\colon K^0\to K'$ of free modules of the same rank, since $\chi=0$. At $\alpha=0$ the fibre $F$ is acyclic because $E$ is Ulrich, so $\det\varphi\ne0$; as $\Ah$ is integral, $\det\varphi$ is a nonzerodivisor and $\varphi$ is injective. Consequently $\Phi(F)\cong T[-1]$ with $T=R^1p_{\Ah*}\mathcal F$ a torsion sheaf that has the local resolution $0\to K^0\to K'\to T\to0$.

The zeroth Fitting ideal of $T$ is locally generated by $\det\varphi$ and is independent of the presentation \cite[\S20.2]{Eis95}; it defines an effective Cartier divisor $D_E$ whose support is the set of $\alpha$ where $\varphi(\alpha)$ is not invertible, that is, $V^0(F)$. By Theorem~\ref{thm:surface-classification}, $V^0(F)\ne\emptyset$, so $D_E\ne0$. Since $T$ has homological dimension at most one at every point and $\Ah$ is smooth and projective, it has a global resolution $0\to F_1\xrightarrow{\ \phi\ }F_0\to T\to0$ by vector bundles of the same rank: take a surjection onto $T$ from a direct sum $F_0$ of line bundles and let $F_1$ be its kernel, which is locally free. Locally $\phi$ is another square presentation of $T$, so the section $\det\phi$ of $\det F_0\otimes\det F_1^{-1}$ generates the zeroth Fitting ideal of $T$. Hence its divisor is $D_E$, and $[D_E]=c_1(F_0)-c_1(F_1)=c_1(T)$.

On an abelian surface $\chi(G,G)=2\rk(G)\ch_2(G)-c_1(G)^2$ for every object $G$ of the derived category, since the Todd class is trivial. As $\Phi$ is an equivalence, $\chi(T,T)=\chi(F,F)$. Here $\rk T=0$, while $\rk F=2$, $c_1(F)\equiv H$ and $\ch_2(F)=\ch_2(E)-c_1(E)\cdot H+H^2=2H^2-3H^2+H^2=0$ by Proposition~\ref{prop:chern}. Therefore $D_E^2=c_1(T)^2=c_1(F)^2=2d$.

Finally, $\Ah$ is isogenous to $A$, so $\NS(\Ah)$ has rank one and a nonzero effective divisor on $\Ah$ is ample. An ample effective divisor on a smooth projective surface is connected \cite[Corollary~III.7.9]{Har77}. The last assertion follows from Theorem~\ref{thm:surface-classification}(i).
\end{proof}

\subsection{An explicit construction}
We now construct the bundles of Theorem~\ref{thm:surface-classification}. Throughout, $A$ is an abelian surface with $\NS(A)=\Z[L]$, $L$ ample, $L^2=2d$, so that $h^0(A,L\otimes R)=d$ and $H^i(A,L\otimes R)=0$ for $i>0$ and every $R\in\Pic^0(A)$.

\begin{lemma}\label{lem:integral-members}
For every $R\in\Pic^0(A)$ the linear system $|L\otimes R|$ is nonempty, and each of its members is an integral curve.
\end{lemma}

\begin{proof}
Nonemptiness holds because $d\ge1$. Let $C=\sum m_iC_i$ with $C_i$ prime and $m_i\ge1$. Each $C_i$ is numerically equivalent to $a_iL$ with $a_i\in\Z$, and $2d\,a_i=C_i\cdot L>0$ because $L$ is ample, so $a_i\ge1$. From $\sum m_ia_i=1$ we get a single component with $m_1=a_1=1$.
\end{proof}

\begin{lemma}\label{lem:evaluation}
Let $C$ be an integral projective curve, $N$ a line bundle on $C$, and $W\subset H^0(C,N)$ a subspace of dimension $n\ge0$. The tuples $(x_1,\dots,x_n)\in C^n$ for which the evaluation map $W\to\bigoplus_jN\otimes k(x_j)$ is an isomorphism form a nonempty open subset of $C^n$, and for such tuples the $x_j$ are pairwise distinct.
\end{lemma}

\begin{proof}
The condition is the nonvanishing of a determinant in local trivializations, so it is open; we show it is nonempty by induction on $n$, the case $n=0$ being trivial. Since $C$ is integral, a nonzero $w\in W$ vanishes at finitely many points; choose $x_1$ with $w(x_1)\ne0$. Then $W_1=\{w\in W:w(x_1)=0\}$ has dimension $n-1$, and by induction there are $x_2,\dots,x_n$ for which evaluation on $W_1$ is an isomorphism. An element of $W$ vanishing at $x_1,\dots,x_n$ lies in $W_1$ and is therefore zero, so evaluation on $W$ is injective, hence an isomorphism. If $x_i=x_j$ with $i\ne j$, the image lies in a proper subspace, so the map is not surjective.
\end{proof}

\begin{theorem}\label{thm:surface-existence}
Let $P,Q\in\Pic^0(A)$ be nontrivial, put $M=L\otimes Q\otimes P^{-1}$, and let $C\in|M|$. There is a nonempty open subset $\Omega\subset C^d$ such that for every $(x_1,\dots,x_d)\in\Omega$ and $Z=\{x_1,\dots,x_d\}$:
\begin{enumerate}[label=\textup{(\roman*)}]
\item there is an exact sequence \eqref{eq:surface-ext} with $E$ locally free;
\item every locally free $E$ in such a sequence is a $\mu_L$-stable rank-two Ulrich bundle with $c_1(E)\equiv3H$ and $c_2(E)=5d$.
\end{enumerate}
In particular, rank-two Ulrich bundles exist on $(A,L)$ for every ample generator $L$ of $\NS(A)$.
\end{theorem}

\begin{proof}
By Lemma~\ref{lem:integral-members}, $C$ is an integral curve, and $\OO_A(C)\cong M$.

\emph{Three subspaces.} From $0\to\OO_A\to M\to M|_C\to0$, $h^0(\OO_A)=1$ and $H^1(A,M)=0$, the image $V$ of $H^0(A,M)\to H^0(C,M|_C)$ has dimension $d-1$, and the kernel is spanned by an equation $s_C$ of $C$. The kernel of $H^0(A,L\otimes Q)\to H^0(C,L\otimes Q|_C)$ is $H^0(A,L\otimes Q\otimes M^{-1})=H^0(A,P)=0$, so the image $W_Q$ has dimension $d$. Likewise the kernel of $H^0(A,L\otimes P^{-1})\to H^0(C,L\otimes P^{-1}|_C)$ is $H^0(A,Q^{-1})=0$, and the image $W_P$ has dimension $d$.

\emph{Choice of points.} Let $\Omega\subset C^d$ be the set of tuples such that evaluation at $x_1,\dots,x_d$ is an isomorphism on $W_Q$ and on $W_P$, and, for each $i$, evaluation at $\{x_j\}_{j\ne i}$ is an isomorphism on $V$. By Lemma~\ref{lem:evaluation}, applied to $W_Q$, $W_P$ and, through the projections $C^d\to C^{d-1}$ forgetting the $i$-th factor, to $V$, the set $\Omega$ is a finite intersection of nonempty open subsets of the irreducible variety $C^d$, hence nonempty and open. Fix $(x_i)\in\Omega$. The points are distinct, so $Z=\{x_1,\dots,x_d\}$ is a reduced subscheme of length $d$, and it is a local complete intersection in $A$.

\emph{Cayley--Bacharach.} Let $Z'=Z\setminus\{x_i\}$ and let $\sigma\in H^0(A,M)$ vanish on $Z'$. Then $\sigma|_C\in V$ vanishes at the $d-1$ points of $Z'$, so $\sigma|_C=0$ by the choice of $\Omega$. Hence $\sigma\in\C s_C$, and $\sigma$ vanishes on $C\supset Z$. Thus $Z$ has the Cayley--Bacharach property with respect to $M=M\otimes\omega_A$.

\emph{Serre correspondence.} By \cite[Theorem~5.1.1]{HL10} there is an extension $0\to\OO_A\to F\to M\otimes\I_Z\to0$ with $F$ locally free of rank two. Since $M\otimes L\otimes P=L^2\otimes Q$, the bundle $E=F\otimes L\otimes P$ fits into \eqref{eq:surface-ext}. This proves (i).

\emph{Ulrich property.} Let $E$ be any locally free bundle in \eqref{eq:surface-ext}. The map $r_Q$ factors as $H^0(A,L\otimes Q)\xrightarrow{\ \sim\ }W_Q\to\bigoplus_i(L\otimes Q)\otimes k(x_i)$, and the second map is an isomorphism by the choice of $\Omega$; so $r_Q$ is an isomorphism, and likewise $r_P$. Theorem~\ref{thm:surface-criterion} shows that $E$ is Ulrich, and Corollary~\ref{cor:stable} that it is $\mu_L$-stable. Finally $c_1(E)\equiv H+2H=3H$ and $c_2(E)=(H\cdot 2H)+\length(Z)=4d+d=5d$, in agreement with Proposition~\ref{prop:chern}.
\end{proof}

\begin{corollary}\label{cor:principal}
Let $(A,\Theta)$ be a principally polarized abelian surface with $\NS(A)=\Z[\Theta]$, and put $L=\OO_A(\Theta)$. For $R\in\Pic^0(A)$ let $\Theta_R$ be the unique divisor in $|L\otimes R|$. Let $P,Q\in\Pic^0(A)$ be nontrivial and $x\in\Theta_{QP^{-1}}\setminus(\Theta_Q\cup\Theta_{P^{-1}})$; such points exist. Then there is a locally free extension
\[
0\to L\otimes P\to E\to L^2\otimes Q\otimes\I_x\to0,
\]
and every such $E$ is a $\mu_L$-stable rank-two Ulrich bundle with $c_1(E)\equiv3\Theta$ and $c_2(E)=5$.
\end{corollary}

\begin{proof}
Here $d=1$, $V=0$, and $\Omega=\{x\in C: x\notin\Theta_Q\cup\Theta_{P^{-1}}\}$ for $C=\Theta_{QP^{-1}}$, since evaluation at $x$ on the one-dimensional space $W_Q$ (resp.\ $W_P$) is an isomorphism exactly when the section of $L\otimes Q$ (resp.\ $L\otimes P^{-1}$) does not vanish at $x$. As $P\ne\OO_A$ and $Q\ne\OO_A$, the line bundle $Q\otimes P^{-1}$ differs from $Q$ and from $P^{-1}$, so $\Theta_{QP^{-1}}$ is an integral curve (Lemma~\ref{lem:integral-members}) different from $\Theta_Q$ and $\Theta_{P^{-1}}$. It meets each of them in at most $\Theta^2=2$ points, so $\Omega$ is nonempty. Apply Theorem~\ref{thm:surface-existence}.
\end{proof}

\begin{remark}\label{rem:CNY}
If $L$ is very ample, $P=Q=\eta$ and $C\in|L|$ is smooth, Theorem~\ref{thm:surface-existence} gives back Beauville's bundles \cite{Bea16}. When $\NS(A)=\Z[L]$, the existence of rank-two Ulrich bundles for every ample $L$ also follows from \cite[Theorem~1.10]{CNY25}: the vector $(2,H,0)$ has no isotropic decomposition in their sense, because an isotropic summand $(r_2,\xi_2,a_2)$ with $r_2>0$ and $a_2<0$ would satisfy $\xi_2^2=2r_2a_2<0$, which is impossible when $\NS(A)=\Z[L]$. Theorem~\ref{thm:surface-existence} gives the bundles explicitly, together with a minimal-defect subsheaf.
\end{remark}

By Corollary~\ref{cor:no-ulrich}, none of these bundles is line-filtered. This does not follow from the existence of one nonempty defect, since the defect depends on the chosen subsheaf.

\section{Changing the polarization}\label{sec:polarization}

\subsection{Nonprimitive polarizations}
In this subsection we assume only $\NS(A)=\Z[B]$ with $B$ ample, and write $L=B^m\otimes R$ with $m\ge1$ and $R\in\Pic^0(A)$. Block decompositions are taken with respect to $B$.

\begin{proposition}\label{prop:nonprimitive}
Let $E=\bigoplus B^a\otimes P\otimes U_{a,P}$ be line-filtered. Then $E$ is ACM with respect to $L$ if and only if every occurring pair $(a,P)$ satisfies $m\nmid a$, or $m\mid a$ and $P\otimes R^{-a/m}\ne\OO_A$. No nonzero line-filtered bundle is Ulrich with respect to $L$.
\end{proposition}

\begin{proof}
The twist of a block by $L^t$ is $B^{a+mt}\otimes P\otimes R^t\otimes U_{a,P}$. If $a+mt\ne0$ it has no intermediate cohomology; this holds for all $t$ when $m\nmid a$. Otherwise $t=-a/m$ is the only exponent-zero twist, with degree-zero factor $P\otimes R^{-a/m}$, and Lemma~\ref{lem:unipotent-twists} applies. For the Ulrich assertion, $\chi$ of a block twisted by $L^{-p}$ is a nonzero constant times $(a-mp)^g$; acyclicity for $p=1,2$ would force $a=m=2m$.
\end{proof}

\begin{example}
If $L=B^2$, then $B$ is ACM with respect to $L$ although its degree-zero factor relative to $B$ is trivial: every twist $B\otimes L^t=B^{1+2t}$ has nonzero exponent. Thus the ACM condition depends on the line bundle $L$, not only on the primitive class.
\end{example}

\subsection{The totally ACM condition}
We return to $\NS(A)=\Z[L]$. Following \cite{Bal08}, a vector bundle $E$ is \emph{totally ACM} if $H^i(A,E\otimes B)=0$ for $1\le i\le g-1$ and every $B$ in the set $\gamma'$ of ample line bundles, their inverses, and $\OO_A$. Equivalently, $H^i(A,E\otimes B^t)=0$ for every ample $B$, every $t\in\Z$ and $1\le i\le g-1$.

\begin{proposition}\label{prop:totally}
A nonzero line-filtered bundle $E$ is totally ACM if and only if $E\cong\bigoplus_{P\in S}P\otimes U_P$ with $S\subset\Pic^0(A)\setminus\{\OO_A\}$ finite and $U_P$ unipotent; that is, if and only if $E$ is numerically flat and acyclic.
\end{proposition}

\begin{proof}
A totally ACM bundle is ACM for $L$, so all degree-zero factors are nontrivial by Theorem~\ref{thm:blocks}. If a block $L^a\otimes P\otimes U$ has $a>0$, take $B=L^a\otimes P$ and $t=-1$: the summand becomes $U$, which has $H^1\ne0$. If $a<0$, take $B=L^{-a}\otimes P^{-1}$ and $t=1$. Hence $a=0$ for all blocks. Conversely, for ample $B$ and $t\ne0$, the twists $P\otimes U_P\otimes B^t$ have no intermediate cohomology, and for $t=0$ they are acyclic. The last statement follows from Proposition~\ref{prop:numflat}.
\end{proof}

\begin{remark}
In rank two Proposition~\ref{prop:totally} gives $P\oplus Q$ with $P,Q\ne\OO_A$ and the nonsplit self-extensions of a nontrivial $P$. On an abelian surface this agrees with \cite[Theorem~1.1(iv)]{Bal08}, where no line filtration is assumed; on surfaces Ballico's statement is therefore stronger, while Proposition~\ref{prop:totally} holds in every dimension $g\ge2$.
\end{remark}



\begin{thebibliography}{GKUW26}

\bibitem[Arr07]{Arr07}
E.~Arrondo, \emph{A home-made Hartshorne--Serre correspondence}, Rev. Mat. Complut. \textbf{20} (2007), no.~2, 423--443.

\bibitem[Ati56]{Ati56}
M.~F. Atiyah, \emph{On the Krull--Schmidt theorem with application to sheaves}, Bull. Soc. Math. France \textbf{84} (1956), 307--317.

\bibitem[Bal08]{Bal08}
E.~Ballico, \emph{Rank $2$ totally arithmetically Cohen--Macaulay vector bundles on an abelian surface with $\operatorname{Num}(X)\cong\mathbb Z$}, Rend. Istit. Mat. Univ. Trieste \textbf{40} (2008), 45--53.

\bibitem[Bea16]{Bea16}
A.~Beauville, \emph{Ulrich bundles on abelian surfaces}, Proc. Amer. Math. Soc. \textbf{144} (2016), no.~11, 4609--4611. doi:10.1090/proc/13091.

\bibitem[BL04]{BL04}
C.~Birkenhake and H.~Lange, \emph{Complex Abelian Varieties}, 2nd ed., Grundlehren Math. Wiss., vol.~302, Springer, Berlin, 2004.

\bibitem[CH12]{CH12}
M.~Casanellas and R.~Hartshorne, with an appendix by F.~Gei\ss\ and F.-O.~Schreyer, \emph{Stable Ulrich bundles}, Internat. J. Math. \textbf{23} (2012), no.~8, 1250083. doi:10.1142/S0129167X12500838.

\bibitem[CNY25]{CNY25}
I.~Coskun, H.~Nuer and K.~Yoshioka, \emph{Weak Brill--Noether on abelian surfaces}, Selecta Math. (N.S.) \textbf{31} (2025), no.~3, Paper No.~48.

\bibitem[DPS94]{DPS94}
J.-P. Demailly, T.~Peternell and M.~Schneider, \emph{Compact complex manifolds with numerically effective tangent bundles}, J. Algebraic Geom. \textbf{3} (1994), no.~2, 295--345.

\bibitem[Eis95]{Eis95}
D.~Eisenbud, \emph{Commutative Algebra with a View Toward Algebraic Geometry}, Grad. Texts in Math., vol.~150, Springer, New York, 1995.

\bibitem[ES03]{ES03}
D.~Eisenbud and F.-O. Schreyer, with an appendix by J.~Weyman, \emph{Resultants and Chow forms via exterior syzygies}, J. Amer. Math. Soc. \textbf{16} (2003), no.~3, 537--579. doi:10.1090/S0894-0347-03-00423-5.

\bibitem[Ful98]{Ful98}
W.~Fulton, \emph{Intersection Theory}, 2nd ed., Ergeb. Math. Grenzgeb. (3), vol.~2, Springer, Berlin, 1998.

\bibitem[GKUW26]{GKUW26}
A.~Gross, I.~Kaur, M.~Ulirsch and A.~Werner, \emph{Semi-homogeneous vector bundles on abelian varieties: moduli spaces and their tropicalization}, Moduli \textbf{3} (2026), e8. doi:10.1017/S2949764726100228.

\bibitem[GL87]{GL87}
M.~Green and R.~Lazarsfeld, \emph{Deformation theory, generic vanishing theorems, and some conjectures of Enriques, Catanese and Beauville}, Invent. Math. \textbf{90} (1987), 389--407.

\bibitem[Har77]{Har77}
R.~Hartshorne, \emph{Algebraic Geometry}, Grad. Texts in Math., vol.~52, Springer, New York, 1977.

\bibitem[Har80]{Har80}
R.~Hartshorne, \emph{Stable reflexive sheaves}, Math. Ann. \textbf{254} (1980), 121--176.

\bibitem[HL10]{HL10}
D.~Huybrechts and M.~Lehn, \emph{The Geometry of Moduli Spaces of Sheaves}, 2nd ed., Cambridge Math. Library, Cambridge Univ. Press, Cambridge, 2010.

\bibitem[MB26]{MB26}
S.~Mondal and P.~Barik, \emph{Existence of ACM bundles on polarized abelian varieties}, preprint, arXiv:2606.04741v6, 2026.

\bibitem[Muk78]{Muk78}
S.~Mukai, \emph{Semi-homogeneous vector bundles on an abelian variety}, J. Math. Kyoto Univ. \textbf{18} (1978), no.~2, 239--272.

\bibitem[Muk81]{Muk81}
S.~Mukai, \emph{Duality between $D(X)$ and $D(\hat X)$ with its application to Picard sheaves}, Nagoya Math. J. \textbf{81} (1981), 153--175.

\bibitem[Mum70]{Mum70}
D.~Mumford, \emph{Abelian Varieties}, Tata Inst. Fund. Res. Stud. Math., vol.~5, Oxford Univ. Press, London, 1970.

\bibitem[Yos21]{Yos21}
K.~Yoshioka, \emph{aCM bundles on a general abelian surface}, Arch. Math. (Basel) \textbf{116} (2021), no.~5, 529--539. doi:10.1007/s00013-020-01570-0.

\end{thebibliography}
\end{document}